\documentclass[a4paper,10pt]{article}
\usepackage{amssymb,amsmath,amsthm}
\usepackage{fullpage}
\usepackage{hyperref}
\usepackage{eucal}
\usepackage{color}
\usepackage{stackrel}
\usepackage{graphicx}
\renewcommand{\Im}{\mathrm{Im}\,}
\renewcommand{\Re}{\mathrm{Re}\,}
\newcommand{\ie}{\emph{i.e.}}

\newcommand{\cf}{\emph{cf.}}

\newcommand{\Real}{\mathbb{R}}
\let\R\Real
\newcommand{\Com}{\mathbb{C}}
\let\C\Com
\newcommand{\Nat}{\mathbb{N}}

\newcommand{\Dom}{\mathsf{D}}

\newcommand{\dist}{\mathop{\mathrm{dist}}\nolimits}

\newcommand{\eps}{\varepsilon}
\newcommand{\sii}{L^2}

\newcommand{\der}{\mathrm{d}}

\renewcommand{\P}{\mathbb{P}}
\newcommand{\norm}[1]{\|#1\|}
\newtheorem{Theorem}{Theorem}
\newtheorem{Lemma}[Theorem]{Lemma}
\newtheorem{Proposition}[Theorem]{Proposition}

\numberwithin{Theorem}{section}
\theoremstyle{definition}
\newtheorem{Remark}[Theorem]{Remark}
\numberwithin{equation}{section}

\def\OMIT#1{}
\usepackage[normalem]{ulem}
\definecolor{DarkGreen}{rgb}{0,0.5,0.1} % David

\newcommand\soutD{\bgroup\markoverwith
{\textcolor{DarkGreen}{\rule[.5ex]{2pt}{1pt}}}\ULon}
\newcommand\soutP{\bgroup\markoverwith
{\textcolor{blue}{\rule[.5ex]{2pt}{1pt}}}\ULon}
\newcommand{\Hm}[1]{\leavevmode{\marginpar{\tiny%
$\hbox to 0mm{\hspace*{-0.5mm}$\leftarrow$\hss}%
\vcenter{\vrule depth 0.1mm height 0.1mm width \the\marginparwidth}%
\hbox to
0mm{\hss$\rightarrow$\hspace*{-0.5mm}}$\\\relax\raggedright #1}}}

\begin{document}
%
%-------%
% TITLE %
%-------%
%------------------------------------------%
%------------------------------------------%
\title{\textbf{\LARGE The relativistic rotated harmonic oscillator}}
\author{A.~Balmaseda,$^a$ \
D.~Krej\v{c}i\v{r}{\'\i}k\,$^b$ \ 
and \ J.~M.~P\'erez-Pardo\,$^a$}
\date{\small
\vspace{-5ex}
\begin{quote}
\emph{
\begin{itemize}
\item[$a)$] 
Departamento de Matem\'aticas, Universidad Carlos III de Madrid, Avda. de la Universidad 30, 28911 Legan\'es, Madrid, Spain;
abalmase@math.uc3m.es, jmppardo@math.uc3m.es.%
\item[$b)$] 
Department of Mathematics, Faculty of Nuclear Sciences and 
Physical Engineering, Czech Technical University in Prague, 
Trojanova 13, 12000 Prague 2, Czechia;
david.krejcirik@fjfi.cvut.cz.%
\end{itemize}
}
\end{quote}
15 November 2024
}
\maketitle

\begin{abstract}
\noindent
We introduce a relativistic version of 
the non-self-adjoint operator
obtained by a dilation analytic transformation
of the quantum harmonic oscillator.
While the spectrum is real and discrete,
we show that the eigenfunctions do not form a basis
and that the pseudospectra are highly non-trivial.
\end{abstract}
%

%------------------------------------------%
%------------------------------------------%
%
%------%
% BODY %
%------%
%
%------------------------------------------%
\section{Introduction}
%------------------------------------------%
%
Harmonic oscillations are very common in nature,
ranging from classical pendula to molecular vibrations.
Moreover, general motions can often be approximated 
by a harmonic oscillator for small displacements
around a stable equilibrium.
On the pragmatic side, the importance lies in the availability  
of explicit solutions, including quantum mechanics.

In one dimension, the quantum harmonic motion is described 
in the Hilbert space $\sii(\Real)$
by the Schr\"odinger operator
\begin{equation}\label{harmonic}
  S_0 :=  - \partial_x^2 + x^2 
  \,.
\end{equation}
On its natural (intersection) domain, the operator is self-adjoint 
and its spectrum equals the discrete set 
\begin{equation}\label{spectrum0}
  \sigma(S_0) := \{2n+1\}_{n=0}^\infty
  \,.
\end{equation}
The pseudospectra, described by the norm of the resolvent,
are trivial because of the classical formula 
\begin{equation}\label{trivial}
  \big\|(S_0-z)^{-1}\big\| 
  = \frac{1}{\dist\big(z,\sigma(S_0)\big)} 
  \,.
\end{equation}

In the seminal paper~\cite{Davies_1999a},
motivated by the computation of resonances 
by the dilation analyticity technique,
Davies considered the complex-rotated version 
of the quantum harmonic oscillator
\begin{equation}\label{Davies}
  S_\theta := - e^{-i\theta} \partial_x^2 + e^{i\theta} x^2 
  \qquad \mbox{with} \qquad
  \theta \in (-\mbox{$\frac{\pi}{2}$},\mbox{$\frac{\pi}{2}$})
  \,.
\end{equation}
On the first glance, the transformation looks innocent;
indeed the spectrum is preserved, $\sigma(S_\theta) = \sigma(S_0)$.
In reality, however, the non-self-adjoint operator~$S_\theta$
with $\theta\not=0$ becomes tameless:
the eigenfunctions do not form a basis,
$S_\theta$ is not similar 
(via a bounded and boundedly invertible transformation)
to a normal operator and its pseudospectra are highly non-trivial.
More specifically,
\begin{equation}\label{Davies.wild}
  \lim_{r \to +\infty}
  \big\|\big(S_\theta-r e^{i\vartheta}\big)^{-1}\big\|  
  {=} \infty
  \qquad \mbox{whenever} \qquad
  \vartheta \in (-\theta,\theta) \setminus \{0\}
  \,.
\end{equation}
Note that this wild behaviour is in striking contrast
with the self-adjoint state of the art~\eqref{trivial}.
One conclusion of Davies' paper~\cite{Davies_1999a} 
is that the computation of high-energy resonances 
becomes numerically unstable.
Among the many subsequent studies of the Davies' operator~\eqref{Davies},
let us point out the important contributions 
\cite{Boulton_2002,Davies-Kuijlaars_2004,Pravda-Starov_2006,
Bagarello_2010,
Hitrik-Sjostrand-Viola_2013,Viola_2013,Henry,
KS4,Arnal-Siegl_2023} 
and the survey~\cite{KSTV}.

The objective of this paper is to introduce 
a relativistic version of the rotated harmonic oscillator~\eqref{Davies}. 
There seems to be no consensus on 
the Dirac operator in the spinorial Hilbert space $\sii(\Real)^4$
describing the relativistic quantum harmonic oscillator
\cite{Aldaya-Bisquert-Navarro-Salas_1991,Toyama-Nogami_1999}.
The most popular choice seems to be 
\cite{Moshinsky-Szczepaniak_1989,
Martinez-y-Romero-Nunez-Yepez-Salas-Brito_1995}
\begin{equation}
  H_0 := - i \alpha_1 \partial_x
  - \alpha_2  x
  + m \alpha_3
  \,,
\end{equation}
where $m \geq 0$ is the mass of the particle 
and $\alpha_0,\alpha_1,\alpha_2,\alpha_3$ are 
the standard $4 \times 4$ Dirac matrices. 
Indeed, the anticommutation relations
$
  \{\alpha_\mu, \alpha_\nu\} 
  = 2 \delta_{\mu\nu}
$ 
valid with every $\mu,\nu \in \{0,1,2,3\}$
lead to the desired result that~$H_0$ is a ``square root'' of~$S_0$,
see~\eqref{eq:square}.
Then the complex-rotated version of~$H_0$ reads   
\begin{equation}\label{operator}
  H_\theta := - i \alpha_1 e^{-i\theta/2} \partial_x
  - \alpha_2 e^{i\theta/2} x
  + m \alpha_3
  \qquad \mbox{with} \qquad
  \theta \in (-\mbox{$\frac{\pi}{2}$},\mbox{$\frac{\pi}{2}$})
  \,.
\end{equation}
The goal of this paper is to perform a detailed spectral
and pseudospectral analysis of this operator. The choice of this operator as the relativistic version {of~\eqref{Davies}}
 can be further justified by studying the non-relativistic limit 
{(\cf~Theorem~\ref{thm:nrlimit})}.

The study of Dirac operators has its own interest, 
both from a purely mathematical side 
as well as from the physical one since relevant phenomena in nature
are modelled by them. A prominent example beyond the relativistic systems
is the description of the graphene molecule \cite{Novoselov_2012}. 
Despite the self-adjoint case has been treated extensively in the literature,   
the non-self-adjoint situations have been considered only in few cases \cite{Tkachenko_2001,
Cuenin-Laptev-Tretter_2014,Cuenin_2014,Cuenin_2017,
Enblom_2018,Cuenin-Siegl_2018,FK9,CIKS,
Heriban-Tusek_2022,KN,AFKS,Ancona-Fanelli-Schiavone_2022,
Mizutani-Schiavone_2022,KK10,BKN}.

The motivation for the present study is not merely
our mathematical curiosity, but also the concept
of unconventional representation of quantum observables
by non-self-adjoint operators.
This observation goes back to a 1992 article by physicists \cite{GHS}
concerned with bounded operators similar to self-adjoint ones. However, there are important mathematical intricacies
for unbounded operators, which can be conveniently handled
by the notion of pseudospectra, see \cite{KSTV} for a survey.

It follows from the present results that $H_\theta$
is actually not quasi-self-adjoint (\ie\ similar, via a bounded
and boundedly invertible transformation, to a self-adjoint operator). Our argument is again based on a pseudospectral analysis,
namely, on studying the growth of the norm of the resolvent
for large complex energies. 

%The motivation for the present study is not merely 
%our mathematical curiosity, 
%but also the new concept of unconventional representation 
%of quantum observables by non-self-adjoint operators
%\cite{GHS,KSTV}.
%In particular, it follows from our results that~$H_\theta$
%is not quasi-self-adjoint
%(\ie\ similar, via a bounded and boundedly invertible transformation,
%to a self-adjoint operator).
More generally, despite the growing interest in non-self-adjoint 
Dirac operators, very few explicitly solvable models 
are available (see~\cite{BKN} from the list above for an exception).
In this paper, we offer the community a relativistic analogue of 
the highly influential Davies's oscillator~\eqref{Davies},
which involves both the illusory simplicity due to the availability 
of explicit solutions as well as all the non-self-adjoint peculiarities. 

Let us now present our main results. 
First of all, we realise~$H_\theta$ 
as a closed operator with compact resolvent
(see Theorem~\ref{Thm.closed}).
Contrary to the sectorial Schr\"odinger case~\eqref{Davies}, 
however, the numerical range of~$H_\theta$ coincides 
with the whole complex plane whenever $\theta\not=0$. 
The crucial observation is the supersymmetric relationship
\begin{equation}\label{eq:square}
  H_\theta^2 = (S_\theta + m^2) I_{\Com^4} + i \alpha_1\alpha_2 
  \,,
\end{equation}
where the matrix~$\alpha_1\alpha_2$ is diagonal
{(\cf~\eqref{eq:involutionmatrix})}. 
Consequently,
\begin{equation}\label{eq:spectrum_Dirac}
  \sigma(H_\theta)
  = \left\{-\sqrt{2n+m^2},\sqrt{2n+m^2}\right\}_{n=0}^\infty
  \,.
\end{equation}
Denoting by $\P_n^\pm$
the projectors onto the eigenspaces
of the eigenvalues $\pm \sqrt{2n+m^2}$,
our prime result determines the asymptotics of the norms of 
the eigenprojectors. 
\begin{Theorem}\label{Thm.proj}
One has
$$
  \lim_{n \to \infty} \frac{\log \| \P_n^\pm \|}{n} 
  = \log \sqrt{\frac{1+|\sin\theta|}{1-|\sin\theta|}}
  \,.
$$
\end{Theorem}

Here the self-adjoint realisation $\theta=0$ is not interesting,
for $\| \P_n^\pm \| = 1$ because the eigenfunctions of~$H_0$ 
form an orthonormal basis.
On the other hand,
whenever $\theta\not=0$, the blow-up of the spectral projectors 
implies that the eigenfunctions form neither Riesz nor Schauder basis. 

To further advocate the wild non-self-adjoint
behaviour of~$H_\theta$ whenever $\theta\not=0$,
let us recall the notion of the $\eps$-pseudospectrum ($\eps > 0$)
$$
  \sigma_\eps(H_\theta) := \sigma(H_\theta) \cup
  \left\{ z \not\in\sigma(H_\theta) : \
  \|(H_\theta-z)^{-1}\| > \frac{1}{\eps} 
  \right\}
  \,.
$$
Of course, this set is trivial if $\theta=0$,
in the sense that it is just the $\eps$-tubular neighbourhood  
of the spectrum $\sigma(H_\theta)$,
due to the self-adjointness implying an analogue of~\eqref{trivial}. 
On the other hand, whenever $\theta\not=0$,
we show that
there are points in $\sigma_\eps(H_\theta)$ lying 
``very far'' from the spectrum.
\begin{Theorem}\label{Thm.pseudo}
Given any $\delta>0$, there exist $C_1,C_2 > 0$ such that,
for all $\eps>0$, 
\begin{multline*}
  \left\{
  z \in \Com : \
  |z^2-m^2| \geq C_1 
  \ \land \
  |\arg (z^2 - m^2)| \leq \theta - \delta  
  \ \land \
  |z^2-m^2| > C_2 \, \log\frac{1}{\eps^2}
  \right\}
  \\
  \subset \sigma_\eps(H_\theta) \subset
  \\
  \left\{z \in \mathbb{C} : \ 
  \left(1 - {| \tan \mbox{$\frac{\theta}{2}$} |} \right) 
  {|\Im z|}
  \leq (|\Re z| + m) \,
  {|\tan\mbox{$\frac{\theta}{2}$}|} + \varepsilon\right\}
  \,.
\end{multline*}  
\end{Theorem}

Here the first inclusion is a consequence of~\eqref{eq:square}
and known estimates on the pseudospectra of~$S_\theta$.
As a consequence, we get the following relativistic analogue 
of the Davies' result~\eqref{Davies.wild}:
\begin{equation}\label{wild} 
  \lim_{r \to \pm\infty}
  \big\|\big(H_\theta \mp m -r e^{i\vartheta}\big)^{-1}\big\|  
  = \infty
  \qquad \mbox{whenever} \qquad
  \vartheta \in 
  (-\mbox{$\frac{\theta}{2}$},\mbox{$\frac{\theta}{2}$}) 
  \setminus \{0\}
  \,.
\end{equation}
The second inclusion demonstrates that the first inclusion
is {optimal for small values of~$\theta$}.  

A pictorial representation of the shape of the bounds 
of Theorem~\ref{Thm.pseudo} can be found in Figure~\ref{fig:pseudobounds}.
The region displayed in light gray corresponds to 
the second inclusion of Theorem~\ref{Thm.pseudo}
for $\varepsilon=0.01$, $m=2$ and $\theta = \frac{\pi}{4}$.
The region in dark gray represents 
the first inclusion of Theorem~\ref{Thm.pseudo}
for the same choice of parameters. 

  \begin{figure}[h!]
    \begin{center}
      \includegraphics[width=0.8\textwidth]{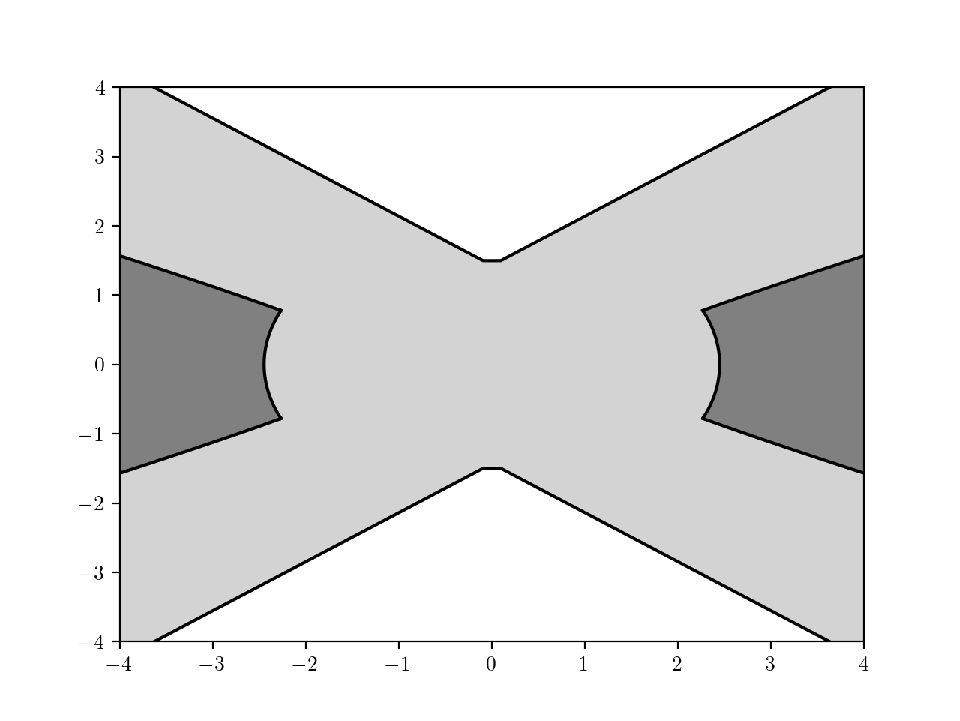}
    \end{center}
    \caption{Representation of a region in the complex plain containing $\sigma_\varepsilon(H_\theta)$ (light gray) and a region definitely contained in $\sigma_\varepsilon(H_\theta)$ (dark gray)
    for $\varepsilon=0.01$, $m=2$ and $\theta = \frac{\pi}{4}$.}\label{fig:pseudobounds}
  \end{figure}

The numerically computed pseudospectra of $H_\theta$ 
for $m=1$ and $\theta = \frac{\pi}{4}$
are depicted in Figure~\ref{fig:pseudospectrum} 
together with the bounds obtained from 
the second inclusion of Theorem~\ref{Thm.pseudo}.
{The (lack of) optimality of Theorem~\ref{Thm.pseudo}
is discussed in more detail at the end of Section~\ref{Sec.pseudo},
see Figure~\ref{Fig.bounds}.}

\begin{figure}[h!]
    \begin{center}
      \includegraphics[width=0.8\textwidth]{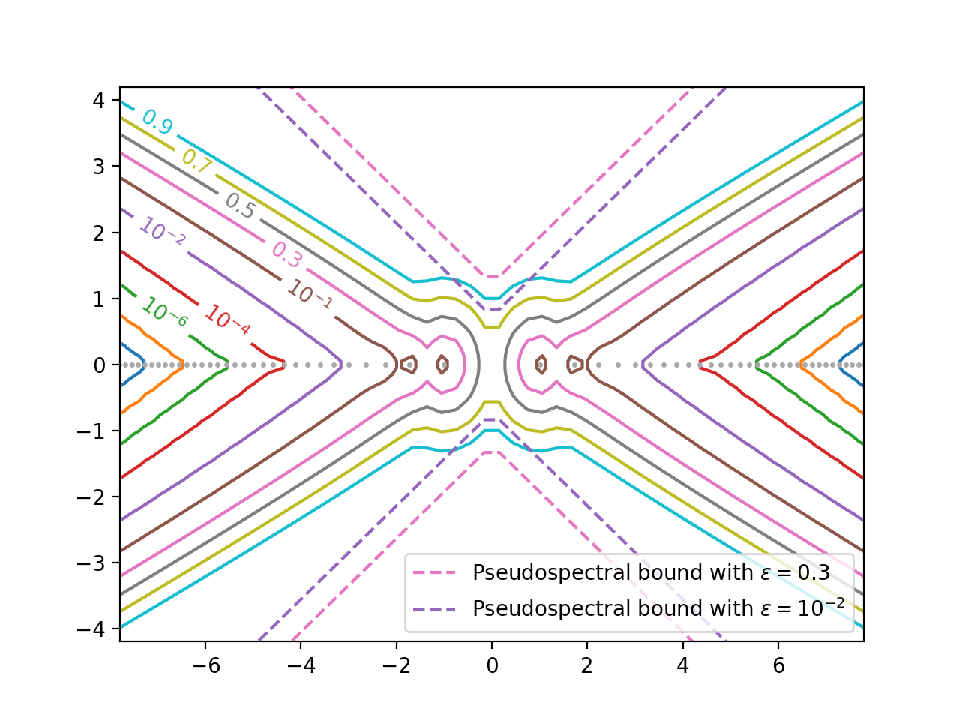}
    \end{center}
    \caption{
      Numerical calculation of the pseudospectrum 
      of $\sigma_\varepsilon(H_\theta)$
      for $m=1$ and $\theta = \frac{\pi}{4}$ and different values of $\varepsilon$.
      The gray dots represent the spectrum, 
      while the dashed lines represent the bounds obtained from 
the second inclusion of Theorem~\ref{Thm.pseudo}
    for different values of $\varepsilon$.
    }
    \label{fig:pseudospectrum}
  \end{figure}

The organisation of this paper is as follows.
In the preliminary Section~\ref{Sec.pre}, 
we provide explicit formulae for the Dirac matrices
and collect known facts about 
the non-relativistic (Schr\"odinger) operator~$S_\theta$.
In Section~\ref{Sec.def}, we define 
the relativistic (Dirac) operator~$H_\theta$
as a closed operator with compact resolvent.
In Section~\ref{Sec.spec}, we determine the eigenvalues
and eigenfunctions of~$H_\theta$ and establish Theorem~\ref{Thm.proj}.
Theorem~\ref{Thm.pseudo} is proved in Section~\ref{Sec.pseudo}.

%------------------------------------------%
\section{Preliminaries}\label{Sec.pre}
%------------------------------------------%
%
In this section we are going to introduce notation and some results that we shall need for the definition 
of the operator~$H_\theta$ from~\eqref{operator}.

\subsection{The Dirac matrices}
The Dirac matrices appearing in~\eqref{operator}
explicitly read (standard representation)
\begin{equation}\label{eq:alphamatrices}
  \alpha_k := 
  \begin{pmatrix}
    0 & \sigma_k \\
    \sigma_k & 0
  \end{pmatrix} 
  \quad \mbox{with} \quad 
  k \in \{1,2,3\}
  \qquad \mbox{and} \qquad
  \alpha_0 :=
   \begin{pmatrix}
    \sigma_0 & 0 \\
    0 & -\sigma_0
  \end{pmatrix} 
  \,,
\end{equation}
where
\begin{equation}\label{eq:sigmamatrices}
  \sigma_1 := 
  \begin{pmatrix}
    0 & 1 \\
    1 & 0
  \end{pmatrix} 
  \,, \qquad
  \sigma_2 := 
  \begin{pmatrix}
    0 & -i \\
    i & 0
  \end{pmatrix} 
  \,, \qquad
  \sigma_3 := 
  \begin{pmatrix}
    1 & 0 \\
    0 & -1
  \end{pmatrix} 
  \,, \qquad
  \sigma_0 := 
  \begin{pmatrix}
    1 & 0 \\
    0 & 1
  \end{pmatrix} 
  = I_{\Com^2}
  \,.
\end{equation}
are the Pauli matrices.
The Dirac (as well as Pauli) matrices are Hermitian 
and the anticommutation relations
$
  \{\alpha_\mu, \alpha_\nu\} 
  := \alpha_\mu\alpha_\nu + \alpha_\nu\alpha_\mu  =2 \delta_{\mu\nu}I_{\Com^4}
$ 
hold
with every $\mu,\nu \in \{0,1,2,3\}$.
These anticommutation properties lead to~\eqref{eq:square},
where the matrix $ i \alpha_1\alpha_2$ reads
\begin{equation}\label{eq:involutionmatrix}
  i \alpha_1\alpha_2  = 
  \begin{pmatrix}
    -\sigma_3 & 0  \\
    0 & -\sigma_3
  \end{pmatrix} 
  \,.
\end{equation}

\subsection{The non-relativistic harmonic oscillator}
In view of~\eqref{eq:square}, 
the Schr\"odinger operator~$S_\theta$ from~\eqref{Davies} 
plays a crucial role in the present analysis.
Let us collect from \cite[Sec.~VII.D]{KSTV}
some basic properties of~$S_\theta$ needed later.

Equipping~$S_\theta$ with the {$\theta$-independent} domain
\begin{equation}\label{eq:domainStheta}
  \Dom({S_\theta}) := 
  H^2(\Real) \cap \sii(\Real,x^4 \, \der x)
  \,,
\end{equation}
%form domain
%$$
%  \mathcal{D} := 
%  H^1(\Real) \cap \sii(\Real,x \, \der x)
%  \,,
%$$
it is realised as an m-sectorial operator with compact resolvent
provided that $|\theta| < \pi/2$. 
It is self-adjoint if, and only if, $\theta=0$.
Indeed, the adjoint operator satisfies $S_\theta^* = S_{-\theta}$.
The numerical range of~$S_\theta$ lies in the symmetric sector
$
  \{z \in \Com: \
  |\Im z| \leq |\tan\theta| \, \Re z 
  \ \land \
  \Re z \geq 0
  \}
$.
The form domain of this operator is 
\begin{equation}\label{eq:formdomain}
  \mathcal{D} := 
  H^1(\Real) \cap \sii(\Real,x^2 \, \der x)
  \,.
\end{equation}

The spectrum of~$S_\theta$ is purely discrete and coincides with 
the spectrum~\eqref{spectrum0}
of the unrotated operator~\eqref{harmonic}, 
\begin{equation}\label{spectrum.harmonic}
  \sigma(S_\theta) = \sigma(S_0) = \{2n+1\}_{n \in \Nat}
  \,.
\end{equation}
(The set of natural numbers~$\Nat$ contains zero in our convention.) 
The eigenvalues are (algebraically) simple.

The eigenfunctions are the complex rotation of the eigenfunctions 
of the self-adjoint harmonic oscillator~$S_0$. 
More specifically, let 
$
  \{\phi_n\}_{n\in\Nat}\subset \Dom(S_0)
 % = H^2(\Real) \cap \sii(\Real,x^2 \, \der x)
$ 
be the orthonormal basis of $\sii(\Real)$
formed by the eigenfunctions of~$S_0$, 
\ie, $S_0\phi_n=(2n+1)\phi_n$. 
Then, for every $n \in \Nat$,
\begin{equation}\label{rotated_eigenfunction}
\begin{aligned}
  S_\theta\varphi_n&=(2n+1)\varphi_n
  \qquad \mbox{with} \qquad
  & \varphi_n(x) &:= \phi_n(e^{i\theta/2}x)
  \,{,}
  \\
  S_\theta^*\tilde\varphi_n&=(2n+1)\tilde\varphi_n
  \qquad \mbox{with} \qquad
  & \tilde\varphi_n(x) &:= \phi_n(e^{-i\theta/2}x)
  \,.
\end{aligned}  
\end{equation}
Moreover, $\{\varphi_n,\tilde\varphi_n\}_{n \in \Nat}$ 
is a biorthonormal basis of $\sii(\Real)$,
\ie, $(\tilde{\varphi}_n,\varphi_m) = \delta_{nm}$
for every $n,m \in \Nat$. 

Despite the fact that the eigenvalues remain real,
the eigenfunctions of~$S_\theta$ have wild basis properties
unless $\theta = 0$.
Indeed, defining the spectral projectors 
$P_n := \varphi_n (\tilde\varphi_n,\cdot)$,
it follows that   
\begin{equation}\label{projectors}
  \lim_{n\to\infty} \frac{\log \|P_n\|}{n}
  = \lim_{n\to\infty} \frac{\log \|\varphi_n\|^2}{n}
  = \log \sqrt{ \frac{1+|\sin\theta|}{1-|\sin\theta|} }
  \,.
\end{equation}
Consequently, the eigenfunctions of~$S_\theta$ with $\theta\not=0$
form neither Riesz nor Schauder basis.

Similarly, the pseudospectra of~$S_\theta$ 
are highly non-trivial unless $\theta = 0$.
Indeed, given any $\delta>0$, there exist $C_1,C_2 > 0$ such that,
for all $\eps>0$, 
\begin{equation}\label{pseudo}
  \left\{
  z \in \Com : \
  |z| \geq C_1 
  \ \land \
  |\arg (z)| \leq \theta - \delta  
  \ \land \
  |z| > C_2 \, \log\frac{1}{\eps}
  \right\}
  \subset \sigma_\eps(S_\theta)
  \,.
\end{equation}  
{In particular, the wild behaviour~\eqref{wild} follows.
Moreover, it is known 
that this inclusion is optimal in the sense 
that $|\vartheta| = |\theta|$ is the transition angle
for the pseudospectral properties of~$S_\theta$.
Indeed,
\begin{equation}\label{Davies.nowild} 
  \limsup_{r \to +\infty}
  \big\|\big(S_\theta -r e^{i\vartheta}\big)^{-1}\big\|  
  < \infty
  \qquad \mbox{whenever} \qquad
  |\vartheta| \in \big[|\theta|,\mbox{$\frac{\pi}{2}$}\big)  
  \,.
\end{equation}
We refer to Boulton's conjecture \cite{Boulton_2002},
its proof in \cite{Pravda-Starov_2006},
book \cite[Prop.~14.13]{Helffer}
and the general approach~\cite{Arnal-Siegl_2023}.}
 
\section{The relativistic harmonic oscillator}\label{Sec.def}
\subsection{Definition}
To introduce~$H_\theta$ as a closed operator with non-empty 
resolvent set, we set for a moment $m=0$ and write
\begin{equation}\label{eq:H_theta_AB}
  H_\theta = A + B
  \,,
  \qquad \mbox{where} \qquad
\begin{aligned}
  A &:= \cos(\mbox{$\frac{\theta}{2}$}) \,
  (- i \alpha_1 \partial_x - \alpha_2 x)
  \,, \qquad
  \\
  B &:= \sin(\mbox{$\frac{\theta}{2}$}) \,
  (- \alpha_1 \partial_x - i\alpha_2 x)
  \,,
\end{aligned}
\end{equation}
are the symmetric and anti-symmetric parts of~$H_\theta$, respectively. 

\begin{Lemma}
$A$ with $\Dom(A) := C_0^\infty(\Real)^4$ is essentially self-adjoint.  
\end{Lemma}
\begin{proof}
Since the factor $\cos(\mbox{$\frac{\theta}{2}$})$ 
plays no role in the property, we omit it throughout this proof.

We start with an elliptic regularity argument. 
Let $\Omega\subset\R$ be a bounded open set. 
Assume that $\Psi\in\sii(\R)^4$ and that $(A+i)\Psi=0$. Then, by the properties of the $\alpha$ matrices, it follows that
\begin{equation*}
-i\alpha_1\partial_x\Psi-\alpha_2x\Psi+i\Psi=0
\quad\Leftrightarrow\quad
\partial_x\Psi-i\alpha_1\alpha_2x\Psi-\alpha_1\Psi=0
\quad\Leftrightarrow\quad
\partial_x\Psi = (i\alpha_1\alpha_2x+\alpha_1)\Psi
\,.
\end{equation*}
Notice that for any $x\in\Omega$, 
$i\alpha_1\alpha_2x+\alpha_1\in \C^{4 \times 4}$ is a Hermitian matrix. 
It follows that
\begin{equation*}%\label{eq:derivative_bound}
%\begin{aligned}
\norm{\partial_x\Psi} 
    = \norm{(i\alpha_1\alpha_2x+\alpha_1)\Psi}
    \leq \norm{x\Psi} + \norm{\Psi}
    \leq K\,\norm{\Psi}
%\end{aligned}
\end{equation*}
with a positive constant~$K$,
where we have used the boundedness of~$\Omega$. 
Hence, since $\Psi\upharpoonright\Omega \in \sii(\Omega)^4$, 
the above inequality shows that 
$\Psi\upharpoonright\Omega\in H^1(\Omega)^4$ 
and therefore $\Psi\in H^1_{\mathrm{loc}}(\R)^4$. 

Now suppose, in addition, that $\Psi\in H^k_{\mathrm{loc}}(\R)^4$
with $k \in \Nat$. 
Then
\begin{equation*}
\begin{aligned}%\label{eq:kth_derivative_bound}
\norm{\partial^{k+1}_x\Psi} 
    = \norm{\partial^{k}_x(i\alpha_1\alpha_2x\Psi)+\alpha_1\partial^{k}_x\Psi}
    \leq \norm{x\partial^{k}_x\Psi} + k \, \norm{\partial^{k-1}_x\Psi} + \norm{\partial^{k}_x\Psi}
    \leq K\ (\norm{\partial^k_x\Psi} + \norm{\Psi})
\end{aligned}
\end{equation*}
with another positive constant~$K$, 
where we have used that 
$\norm{\partial^{k-1}_x\Psi}\leq C \, (\norm{\partial^k_x\Psi} + \norm{\Psi})$
with a positive constant~$C$. 
Therefore $\Psi\in H^k_{\mathrm{loc}}(\R)^4$ implies that 
$\Psi\in H^{k+1}_{\mathrm{loc}}(\R)^4$. By induction we get that 
$\Psi\in H^k_{\mathrm{loc}}(\R)^4$ for any $k\in\mathbb{N}$ 
and therefore $\Psi\in C^\infty(\R)^4$ by the Sobolev embedding. 

We continue now by an argument analogous to the proof of \cite[Thm.~4.3]{Thaller}.
Let $f\in C^\infty_0(\R)$ be such that $f(x) = 1$ for $|x| \leq 1$ 
and let $f_n(x) = f(x/n)$ with $n \geq 1$. 
From the elliptic regularity we have that $\Psi\in\sii(\R)^4$ and $(A+i)\Psi=0$ imply that $\Psi\in C^\infty(\R)^4$. It follows that:
\begin{equation*}%\label{eq:regularisation_psi}
(A+i)(f_n\Psi)(x) 
= -i f_n'(x) \alpha_1\Psi(x) 
= -\frac{i}{n} f'(x/n) \alpha_1\Psi(x).
\end{equation*}
It holds that $\displaystyle \lim_{n\to \infty} \norm{f_n\Psi - \Psi} =0$. 
Since $f_n\Psi\in C^\infty_0(\R)$ and $A$ is a symmetric operator 
on this set, one gets
$$
  \norm{f_n\Psi}^2
  \leq  \norm{f_n\Psi}^2 +\norm{A(f_n\Psi)}^2 
  = \norm{(A+i)(f_n\Psi)}^2
  \leq \frac{1}{n^2} \, \|f'\|_\infty^2 \, \norm{\Psi}^2
  \,.
$$ 
Consequently, $\displaystyle \lim_{n\to \infty} \norm{f_n\Psi} =0$ 
and therefore $\Psi=0$. An analogous argument for $(A-i)\Psi=0$ shows that both deficiency indices vanish, from which the essential self-adjointness follows.
\end{proof}

Let us keep the same notation~$A$ for
the (self-adjoint) closure of~$A$ initially defined on $C_0^\infty(\Real)^4$.  
Integrating by parts, one can check the identity
\begin{equation*}
  \|A\psi\|^2 =  \cos^2(\mbox{$\frac{\theta}{2}$}) \,
  \left(
  \|\psi'\|^2 + \|x\psi\|^2 + \widetilde{\|\psi\|^2}
  \right)
\end{equation*}
for every $\psi \in C_0^\infty(\Real)^4$,
where
$
  \widetilde{\|\psi\|^2}
  := -\|\psi_1\|^2 + \|\psi_2\|^2 - \|\psi_3\|^2 + \|\psi_4\|^2 
$.
Taking $k \in \Real$, one has 
\begin{equation}\label{ineq.A}
  \|(A+ik)\psi\|^2 = \|A\psi\|^2 + k^2 \|\psi\|^2    
  \begin{cases}
    \leq  \|\psi'\|^2 + \|x\psi\|^2 + (k^2+1) \|\psi\|^2 
    \,,
    \\
    \geq \cos^2(\mbox{$\frac{\theta}{2}$}) \,
  \left(
  \|\psi'\|^2 + \|x\psi\|^2 
  \right)
  + \big( k^2 - \cos^2(\mbox{$\frac{\theta}{2}$}) \big) \,  \|\psi\|^2 
  \,.
  \end{cases} 
\end{equation}
It follows that the graph norm of $A+ik$ is equivalent 
to the graph norm of~$S_0^{1/2}$ whenever 
$k^2 \geq \cos^2(\mbox{$\frac{\theta}{2}$})$ and $|\theta| < \pi$. 
Consequently, $\Dom(A) = \mathcal{D}^4$ with $\mathcal{D}$ 
{being} the {$\theta$-independent} form domain of~{$S_\theta$} 
given in~{\eqref{eq:formdomain}}.
Moreover, $A+ik : \Dom(A) \to \sii(\Real)^4$ is bijective
under these conditions.
Similarly,
\begin{equation}\label{ineq.B}
\begin{aligned}
  \|B\psi\|^2 
  &=  \sin^2(\mbox{$\frac{\theta}{2}$}) \,
  \left(
  \|\psi'\|^2 + \|x\psi\|^2 - \widetilde{\|\psi\|^2}
  \right)
  \\
  &\leq
  \sin^2(\mbox{$\frac{\theta}{2}$}) \,
  \left(
  \|\psi'\|^2 + \|x\psi\|^2  
  \right)
  + \sin^2(\mbox{$\frac{\theta}{2}$}) \, \|\psi\|^2  
  \,.
\end{aligned}  
\end{equation}
Comparing the lower bound of~\eqref{ineq.A} 
with the estimate~\eqref{ineq.B},
it follows that 
$\|B\psi\| \leq b \, \|(A+ik)\psi\|$ with $b < 1$ 
provided that 
\begin{equation*}
  \sin^2(\mbox{$\frac{\theta}{2}$}) < \cos^2(\mbox{$\frac{\theta}{2}$}) 
  \qquad \mbox{and} \qquad
  \sin^2(\mbox{$\frac{\theta}{2}$}) < k^2 - \cos^2(\mbox{$\frac{\theta}{2}$})
  \,.
\end{equation*}
These conditions are equivalent to $|\theta| < \frac{\pi}{2}$ and $k^2 > 1$,
respectively.
By \cite[Thm.~IV.1.16]{Kato}, it follows that $A+ik+B$ with 
$\Dom(A+B) := \Dom(A) = \mathcal{D}^4$
is closed and bijective with compact inverse. 
Since adding~$ik$ as well as $m\alpha^3$ represent 
just bounded perturbations, 
we conclude with the following result 
(where~$m$ is possibly non-zero).

\begin{Theorem}\label{Thm.closed}
If $|\theta| < \frac{\pi}{2}$, then $H_\theta$ 
with $\Dom(H_\theta) := \mathcal{D}^4$
is a closed operator with compact resolvent.  
\end{Theorem}
\begin{Remark}
The exclusion of the points $\theta = \pm \frac{\pi}{2}$ is reasonable,
because then even the Schr\"odinger operator $S_\theta$ is not well defined. 
On the other hand, the other points $\theta \in (\frac{\pi}{2},\frac{3\pi}{2})$ 
can be included by noticing the similarity relation
$\alpha_1 H_\theta \alpha_1 = i H_{\theta+\pi}$
valid for $m=0$.
From now on, we restrict to $|\theta| < \frac{\pi}{2}$,
without loss of generality.
\end{Remark}

\subsection{The adjoint operator}
A straightforward computation shows that
the formal adjoint of~$H_\theta$ is given by $H_{-\theta}$.
In this subsection, we prove that, indeed,
$
  H_\theta^* = H_{-\theta} 
$.
For doing so, we first establish the following result.

\begin{Lemma}\label{Lem.resolvent}
  $z\in\rho(H_{\theta})$ if, and only if, $\bar{z}\in\rho(H_{-\theta})$.
\end{Lemma}

\begin{proof}
It is enough to prove the implication $z\in\rho(H_{\theta}) \Rightarrow \bar{z}\in\rho(H_{-\theta})$. 
Take $z\in\rho(H_{\theta})$ and suppose that $\Psi \in \ker(H_{-\theta}-\bar{z})$. Then 
$$\left((H_{-\theta}-\bar{z})\Psi,\Phi\right) = 0 \,,
\quad\forall \Phi \in \Dom(H_\theta) \,.$$
Therefore 
$$\left(\Psi,(H_{\theta}-{z})\Phi\right) = 0 \,,
\quad\forall \Phi \in \Dom(H_\theta) \,,$$
and since $H_{\theta}-{z}$ is surjective, we get that $\Psi=0$, which proves that $H_{-\theta}-\bar{z}$ is injective. 

Now we prove that its range is dense. Let $\Psi\in \operatorname{ran}(H_{-\theta}-\bar{z})^{\bot}\cap\Dom({H_\theta})$. Then we have that 
$$((H_\theta-z)\Psi,\Phi)=0 \,, \quad \forall \Phi\in\Dom({H_{-\theta}}) \,.$$
Since $\Dom(H_{-\theta})$ is a dense subset this implies that $\Psi\in\ker{(H_\theta-z)}= \{0\}$. 

Finally, we prove the surjectivity. Let $\Psi\in\overline{\operatorname{ran}(H_{-\theta}-\bar{z})}$. Then there exists a sequence $\{\Phi_n\}_{n\in\Nat}\subset{\Dom(H_{-\theta})}$ such that $\norm{\Psi-(H_{-\theta}-\bar{z})\Phi_n}\to0$ as $n \to \infty$. 
Taking into account that $H_\theta -z$ is bounded away from the origin, 
\ie, there exists $K>0$ such that $\norm{(H_\theta -z)\xi} \geq K \norm{\xi}$ for all $\xi\in\Dom(H_\theta)$, we have that 
$$
%\begin{aligned} 
  \norm{\Phi_n-\Phi_m} = \sup_{\xi\in\Dom(H_\theta)}\frac{|(\Phi_n-\Phi_m,(H_\theta-z)\xi)|}{\norm{(H_\theta-z)\xi}} \leq \frac{1}{K} \,\norm{(H_{-\theta}-\bar{z})(\Phi_n-\Phi_m)}.
%\end{aligned}
$$
Hence $\{\Phi_n\}_{n \in \Nat}$ is a convergent sequence. 
Let $\Phi\in\sii(\Real)^4$ be its limit. 
By the closedness of $H_{-\theta}-\bar{z}$, 
one has that $\Psi = (H_{-\theta}-\bar{z})\Phi$, which completes the proof.
\end{proof}

Now we can prove the ultimate result.
\begin{Theorem}\label{Thm.adjoint}
  $\displaystyle H^*_\theta = H_{-\theta}$.
\end{Theorem}
\begin{proof}
The inclusion $\Dom(H_{-\theta})\subset \Dom(H^*_\theta)$ is trivial. Let $\Psi\in\Dom(H^*_{\theta})$ and $z\in\rho(H_{\theta})$. Then there exists $\chi \in \sii(\Real)^4$ such that  
$$
  (\chi,\Phi) = (\Psi, (H_\theta-z)\Phi)
  \,, \quad\forall \Phi\in\Dom(H_\theta)
  \,.
$$
Since $\bar{z}\in\rho(H_{-\theta})$ by Lemma~\ref{Lem.resolvent},
there exists $\xi\in\Dom(H_{-\theta})$ such that $\chi = (H_{-\theta}-\bar{z})\xi$. Then we get that
$$(\xi-\Psi, (H_\theta-z)\Phi) = 0 \,,\quad \forall \Phi\in\Dom(H_\theta) \,.$$
Hence $\Psi=\xi\in\Dom(H_{-\theta})$, which completes the proof.
\end{proof}

\subsection{Symmetries}\label{sec:symmetries}
An important symmetry is the anti-commutation
of~$H_\theta$ with~$\alpha_0$, \ie,
\begin{equation}\label{symmetry}
  H_\theta = - \alpha_0 H_{\theta} \alpha_0
  \,.
\end{equation}
Note that~$\alpha_0$ is unitary and involutive.
As a consequence, we obtain the following symmetries 
of spectra and pseudospectra of~$H_\theta$ 
with respect to the imaginary axis.
By convention, we set $\sigma_0(\cdot) := \sigma(\cdot)$.

\begin{Proposition}\label{prop:z_iff_-z}
For every $\eps \geq 0$, 
\begin{equation*}
  z \in \sigma_\eps(H_\theta) 
  \quad \Longleftrightarrow \quad
  -z \in \sigma_\eps(H_\theta) 
  \,.
\end{equation*}
\end{Proposition}
\begin{proof}
It follows from~\eqref{symmetry} that if~$\lambda$ is an eigenvalue of~$H_\theta$
(with eigenfunction~$\psi$), 
then~$-\lambda$ is also an eigenvalue of~$H_\theta$
(with eigenfunction~$\alpha_0\psi$). 
The symmetry of pseudospectra follows from the identity
\begin{equation}\label{pseudosymmetry}
  \|(H_\theta-z)^{-1}\| 
  = \|\alpha_0(H_\theta-z)^{-1}\alpha_0\| 
  = \|(-H_\theta-z)^{-1}\| 
  = \|(H_\theta+z)^{-1}\| 
\end{equation}
valid for every $z \in \rho(H_\theta)$.
\end{proof}

{By Theorem~\ref{Thm.adjoint}, $H_\theta$ is self-adjoint
if, and only if, $\theta=0$.
In general, for every~$\theta$, $H_\theta$ is complex-self-adjoint
(in the sense of~\cite{CK3}).
More specifically, $H_\theta$ is $\mathcal{C}$-self-adjoint
in the sense that 
\begin{equation}\label{C-symmetry}
  H_\theta^* = \mathcal{C} H_{\theta} \mathcal{C} 
  \,,
\end{equation}
where $\mathcal{C} := \mathcal{PK}$
with $(\mathcal{P}\psi)(x) := \psi(-x)$ 
and $\mathcal{K}\psi := \bar\psi$ 
being the parity and the ``time-reversal'' operators, respectively. 
Note that $\mathcal{C}$ is an involution. 
If $m=0$, it is possible to find more such complex symmetries,
possibly non-involutive.}

%----------------------%
\section{The spectrum}\label{Sec.spec}
%----------------------%
%
In this section, we study the spectrum of the Dirac operator~$H_{\theta}$ 
and properties of its eigenfunctions. 
Our strategy is to employ the supersymmetric relationship~\eqref{eq:square}.
Let us therefore start with the analysis of the square~$H_\theta^2$.
 
\subsection{The square of the Dirac operator}
The identity~\eqref{eq:square} is equivalent to
\begin{equation}\label{eq:Dirac_Laplacian}
  H^2_\theta = 
  (S_\theta + m^2 - 1)
  \begin{pmatrix}
    1 & 0 & 0 & 0\\
    0 & 0 & 0 & 0\\
    0 & 0 & 1 & 0\\
    0 & 0 & 0 & 0\\
  \end{pmatrix}
  +
  (S_\theta + m^2 + 1)
  \begin{pmatrix}
    0 & 0 & 0 & 0\\
    0 & 1 & 0 & 0\\
    0 & 0 & 0 & 0\\
    0 & 0 & 0 & 1\\
  \end{pmatrix}.
\end{equation}
Hence $H^2_\theta = T_+\oplus T_-$ 
with $T_{\pm} := (S_\theta + m^2 \pm 1)\otimes I_{\C^2}$
with $\Dom(T_\pm) := \Dom(S_\theta)^2$.
Since $H_{\theta}^2$ has compact resolvent,
one has that $\sigma(H_{\theta}^2) = \sigma(T_+)\cup\sigma(T_-)$. 
It follows from~\eqref{spectrum.harmonic} and~\eqref{rotated_eigenfunction} 
that $\sigma(T_-) = \{2n +m^2\}_{n = 0}^\infty$,
where each number represents a doubly degenerate eigenvalue
with corresponding eigenfunctions 
$(\varphi_n, 0)^\top$ and $(0, \varphi_n)^\top$. 
Similarly, $\sigma(T_+) = \{2n +m^2\}_{n=1}^\infty$,
where each number represents a doubly degenerate eigenvalue
with corresponding eigenfunctions
$(\varphi_{n-1}, 0)^\top$ and $(0, \varphi_{n-1})^\top$. 
Consequently, 
$$
  \sigma(H_{\theta}^2) = \{2n +m^2\}_{n \in \Nat}
  \,,
$$ 
where 
\begin{itemize}
  \item for $n=0$, the number~$m^2$ represents 
  a doubly degenerate eigenvalue of~$H_{\theta}^2$
  with corresponding eigenfunctions 
  \begin{equation}\label{eigenfunction_n=0}
    \Phi^1_0 :=
    \begin{pmatrix}
      \varphi_0 \\ 0 \\ 0 \\ 0
    \end{pmatrix},
    \quad
    \Phi^2_0 :=
    \begin{pmatrix}
      0 \\ 0 \\ \varphi_0 \\ 0
    \end{pmatrix};
  \end{equation}
  \item for $n \geq 1$, the number~$2n+m^2$ represents 
  a fourfold degenerate eigenvalue of~$H_{\theta}^2$
  with corresponding eigenfunctions 
  \begin{equation}\label{eigenfunction_nneq0}
    \Phi^1_n :=
    \begin{pmatrix}
      \varphi_n \\ 0 \\ 0 \\ 0
    \end{pmatrix},
    \quad
    \Phi^2_n :=
    \begin{pmatrix}
      0 \\ \varphi_{n-1} \\ 0 \\ 0
    \end{pmatrix},
    \quad
    \Phi^3_n :=
    \begin{pmatrix}
      0 \\ 0 \\ \varphi_n \\ 0
    \end{pmatrix},
    \quad
    \Phi^4_n :=
    \begin{pmatrix}
      0 \\ 0 \\ 0 \\ \varphi_{n-1}
    \end{pmatrix}.
  \end{equation}
\end{itemize}

We claim that the root subspaces are exhausted by the eigenspaces,
so the algebraic multiplicities equal the geometric multiplicities. 
Indeed, suppose that for an eigenvalue $\lambda:=2n+m^2$ of $H^2_\theta$, 
the algebraic multiplicity is greater than the geometric multiplicity. 
Then there exists an eigenfunction 
$
  \Psi = \alpha_1 \Phi^1_n + \alpha_2 \Phi^2_n 
  + \alpha_3 \Phi^3_n + \alpha_4 \Phi^4_n
$
with $\alpha_j \in \Com$, $j \in \{1,\dots,4\}$,
and a generalised eigenfunction $\xi\in\Dom(H^2_\theta)$ such that
$(H^2_\theta-\lambda)\xi = \Psi$.
The latter implies, componentwise, 
$
  (S_\theta - (2n \mp 1))\xi_j = \alpha_i\varphi_p 
$,
where $p\in\{n,n-1\}$. At least one of the components is non-trivial,
so this implies the existence of a generalised eigenfunction 
for the Schr\"odinger operator~$S_\theta$, a contradiction.

\subsection{The Dirac operator}
First of all, let us observe the following spectral mapping property.
\begin{Lemma}\label{prop:z_iff_z2}
  $\lambda\in\sigma(H_\theta)$ if, and only if, $\lambda^2\in\sigma(H^2_\theta)$.
\end{Lemma}
\begin{proof}
If $\lambda\in\sigma(H_\theta)$, 
then it is easy to see that $\lambda^2\in\sigma(H^2_\theta)$.
Vice versa, suppose that~$\Psi$ is an eigenfunction of $H^2_\theta$ with eigenvalue $\lambda^2$. Recall that $\Dom(H^2_\theta) = \Dom(S_\theta)^4\subset \mathcal{D}^4 = \Dom(H_\theta)$. Then we have
\begin{equation*}
  (H^2_\theta-\lambda^2)\Psi= 0
  \quad\Leftrightarrow\quad
  (H_\theta+\lambda)(H_\theta-\lambda)\Psi=0
  \,.
\end{equation*}
Then either $(H_\theta-\lambda)\Psi=0$, 
in which case $\lambda\in\sigma(H_\theta)$, 
or $\xi:=(H_\theta-\lambda)\Psi\neq0$ is such that $(H_\theta+\lambda)\xi=0$,
in which case $-\lambda\in\sigma(H_\theta)$. 
The proof is completed by recalling 
the symmetry of Proposition~\ref{prop:z_iff_-z}.
\end{proof}

From this lemma, we immediately deduce~\eqref{eq:spectrum_Dirac}.
It is also easy to see that the eigenspace of~$H_\theta$ 
corresponding to an eigenvalue $\pm \sqrt{2n+m^2}$ is contained 
in the eigenspace of~$H_\theta^2$ 
corresponding to the eigenvalue $2n+m^2$. 

To determine the eigenspaces, 
we rely on a complex rotated version of the ladder operators 
associated to the algebra of the harmonic oscillator. 
Let $p := ie^{-i\theta/2}\partial_x$ with $\Dom(p) := H^1(\Real)$ 
and $q := e^{i\theta/2}x$ with $\Dom(q) := \sii(\Real,x \, \der x)$.
Notice that $[p,q] = iI$ and $p^2+q^2 = S_\theta$. 
Let us define the ladder operators $L_{\pm} := - p \pm iq$
with $\Dom(L_\pm) = \mathcal{D}$
that satisfy, on a test function $\Phi \in C_0^\infty(\Real)$,
$$
  S_\theta L_{\pm}\Phi 
  = L_{\pm} S_\theta \Phi \pm 2 L_{\pm}\Phi
  \,.
$$
These algebraic properties extended to 
the eigenfunction $\Phi := \varphi_n$ of~$S_\theta$ 
corresponding to the eigenvalue~$2n+1$ imply that
$L_{\pm}\varphi_n$ is proportional to~$\varphi_{n\pm1}$ if $n \geq 1$,
$L_+\varphi_0$ is proportional to~$\varphi_1$ 
and $L_-\varphi_0 = 0$. 
The proportionality factor can be computed by comparison 
with the ladder operators of the self-adjoint harmonic oscillator~$S_0$. 
Indeed, let $\{\phi_n\}_{n\in\Nat}$ be the orthonormal eigenbasis of~$S_0$. 
Then one has that
$$
  (-i\partial_x+ix)\phi_n(x) = \sqrt{2(n+1)} \, \phi_{n+1}
  \,,\qquad  (-i\partial_x-ix)\phi_n(x) = \sqrt{2n} \, \phi_{n-1} \,.
$$
Defining $\hat{x} := e^{i\theta/2}x$ 
and taking into account~\eqref{rotated_eigenfunction}, 
we have that
$$
\begin{aligned}
 L_+\varphi_n(x) &=  \left( -ie^{-i\theta/2}\partial_x+ie^{i\theta/2}x \right)\phi_n(e^{i\theta/2}x)\\
   &= -i \partial_{\hat{x}}\phi_n(\hat{x}) + i\hat{x}\phi_n(\hat{x})\\
   &= \sqrt{2(n+1)} \, \phi_{n+1}(\hat{x}) 
   = \sqrt{2(n+1)} \, \varphi_{n+1}(x).
\end{aligned}
$$
Likewise one can show that $L_-\varphi_n = \sqrt{2n} \, \varphi_{n-1}$.

One can express the operator $H_\theta$ using the ladder operators to get  
\begin{equation}\label{eq:Dirac_ladder}
  H_\theta=
  \begin{pmatrix}
    0 & 0 & m & L_+ \\
    0 & 0 & L_- & -m\\
    m & L_+ & 0 & 0  \\
    L_- & -m& 0 & 0 \\
  \end{pmatrix}.%
\end{equation}
Employing this structure and the discussion about~$S_\theta$ above,
it is straightforward to determine the eigenspaces of~$H_\theta$.

\begin{itemize}
\item
For $n=0$, the eigenfunctions corresponding to the eigenvalues~$\pm m$ 
must lie in the two dimensional subspace  
$\operatorname{span}\{\Phi_0^1,\Phi_0^2\}$. 
If $m > 0$, Proposition~\ref{prop:z_iff_-z} implies that these subspaces
cannot be degenerate.
If $m=0$, the zero eigenvalue is doubly degenerate. 
It is immediate to check that 
$$
  \chi^1_0 := 
  \begin{pmatrix}
    \varphi_0 \\ 0  \\ \varphi_0 \\ 0
  \end{pmatrix}
  , \quad
  \chi^2_0 := 
  \begin{pmatrix}
    \varphi_0 \\ 0 \\ -\varphi_0 \\ 0
  \end{pmatrix} 
  \,
$$
are the eigenfunctions corresponding to the eigenvalue~$m$ and~$-m$,
respectively.
\item 
For $n \geq 1$,
to obtain the eigenspaces for 
the eigenvalues $\pm\sqrt{2n+m^2}$,
we use that $\operatorname{ran} H_\theta \upharpoonright{W_n}\subset W_n$ 
with $W_n = \operatorname{span}\{\Phi^j_n\}_{j=1}^4$,
which follows straightforwardly from~\eqref{eq:Dirac_ladder}. 
Expressing $H_\theta \upharpoonright{W_n}$ in the basis 
$(\Phi^1_n,\Phi^2_n,\Phi^3_n,\Phi^4_n)$ 
given explicitly in~\eqref{eigenfunction_nneq0}, 
we obtain the matrix representation
\begin{equation*}
  H_\theta \upharpoonright {W_n}=
  \begin{pmatrix}
    0 & 0 & m & \sqrt{2n}\\
    0 & 0 & \sqrt{2n} & -m \\
    m & \sqrt{2n} & 0 & 0 \\
    \sqrt{2n} & -m & 0 & 0 \\
  \end{pmatrix}.
\end{equation*}
This is a Hermitian matrix and can be diagonalised with the following results. 
There is a doubly degenerate eigenvalue $\sqrt{2n+m^2}$ 
with eigenvectors
\begin{equation}\label{eq:coefficients_+}
u^1:=N^1
\begin{pmatrix}
  \sqrt{2n}\\
  \sqrt{2n + m ^2} -m\\
  \sqrt{2n}\\
  \sqrt{2n + m ^2} -m\\
\end{pmatrix},\quad
u^2:=N^2
\begin{pmatrix}
  -\sqrt{2n}\\
  \sqrt{2n + m ^2} +m\\
  \sqrt{2n}\\
  -\sqrt{2n + m ^2} -m\\
\end{pmatrix},\quad
\end{equation}
and a doubly degenerate eigenvalue $-\sqrt{2n+m^2}$ 
with eigenvectors
\begin{equation}\label{eq:coefficients_-}
u^3:=N^3
\begin{pmatrix}
  -\sqrt{2n}\\
  \sqrt{2n + m ^2} +m\\
  -\sqrt{2n}\\
  \sqrt{2n + m ^2} +m\\
\end{pmatrix},\quad
u^4:=N^4
\begin{pmatrix}
  \sqrt{2n}\\
  \sqrt{2n + m ^2} -m\\
  -\sqrt{2n}\\
  -\sqrt{2n + m ^2} +m\\
\end{pmatrix}.
\end{equation}
The normalisation constants $N^j \in \Com$
are chosen in such a way that $(u^1,u^2,u^3,u^4)$
is an orthonormal basis of $\C^4$. 
The eigenfunctions 
{$\chi^1_n$ and $\chi^2_n$
(respectively, $\chi^3_n$ and $\chi^4_n$)
of  the operator~$H_\theta$ corresponding to 
the (doubly degenerate) eigenvalues 
$\sqrt{2n+m^2}$ (respectively, $-\sqrt{2n+m^2}$)}
are thus given by 
\begin{equation}\label{explicit1}
\chi^j_n :=
\begin{pmatrix}
  u^j_1\varphi_{n} \\
  u^j_2\varphi_{n-1} \\
  u^j_3\varphi_{n} \\
  u^j_4\varphi_{n-1} \\
\end{pmatrix},
\end{equation}
where $\{u^j_k\}_{k=1}^4$ are the components of the vector $u^j\in\C^4$.
\end{itemize}

As for the square~$H_\theta^2$, we claim that the root subspaces 
of~$H_\theta$ are exhausted by the eigenspaces.    
Indeed, suppose that there exists $\lambda\in\sigma(H_\theta)$, an eigenfunction $\Psi\in\Dom(H_\theta)$ and a generalised eigenfunction $\xi\in\Dom(H_\theta)$. 
Then we have
$$
  (H_\theta-\lambda)\xi= \Psi
  \quad\Rightarrow\quad
  H_\theta(H_\theta-\lambda)\xi= H_\theta\Psi=\lambda\Psi
  \quad\Rightarrow\quad
  H^2_\theta -\lambda(\Psi + \lambda\xi)= \lambda\Psi
  \quad\Rightarrow\quad
  (H^2_\theta -\lambda^2)\xi = 2\lambda\Psi \,.
$$
This implies the existence of a generalised eigenfunction of~$H^2_\theta$, 
which is in contradiction with our previous observations.  
Hence the eigenvalues of~$H_\theta$ are semisimple 
(\ie, their algebraic and geometric multiplicities coincide).

Let $\{\tilde{\varphi}_n\}_{n\in\Nat}$ be the sequence of eigenfunctions 
of the adjoint $S^*_\theta = S_{-\theta}$ 
corresponding to the eigenvalues $\{2n+1\}_{n\in\Nat}$,
which are biorthonormal to $\{\varphi_n\}_{n\in\Nat}$.
Since the above construction holds for every $|\theta|<\frac{\pi}{2}$ and the matrices $H_\theta \upharpoonright W_n$ do not depend on~$\theta$, 
the functions
\begin{equation}\label{explicit2}
\tilde{\chi}^{{j}}_n=
\begin{pmatrix}
  u^{{j}}_1\tilde{\varphi}_{n} \\
  u^{{j}}_2\tilde{\varphi}_{n-1} \\
  u^{{j}}_3\tilde{\varphi}_{n} \\
  u^{{j}}_4\tilde{\varphi}_{n-1} \\
\end{pmatrix}
\end{equation}
{with $n \geq 1$}
determine the eigenfunctions of $H^*_\theta=H_{-\theta}$
corresponding to the {(doubly degenerate)} eigenvalues 
{$\sqrt{2n+m^2}$ for $j=1,2$ and $-\sqrt{2n+m^2}$ for $j=3,4$}. 
By defining 
$
  \Psi^{(n)} 
  := (\varphi_{n},\varphi_{n-1},\varphi_{n},\varphi_{n-1})^\top
  \in W_n
$ 
and analogously $\tilde{\Psi}^{(n)}$, 
one can express the components of the eigenfunctions through 
$(\chi^j_n)_k = u^j_k\Psi^{(n)}_k$ 
and $(\tilde{\chi}^j_n)_k = u^j_k\tilde{\Psi}^{(n)}_k$. 
Using these expressions, we compute
$$
\begin{aligned}
   (\tilde{\chi}^i_n , \chi^j_m)  
   =  \sum_{k = 1}^4 (u^i_k\tilde{\Psi}^{(n)}_k, u^j_k\Psi^{(m)}_k)
   =  \sum_{k = 1}^4 (u^i_k\tilde{\Psi}^{(n)}_k, u^j_k\Psi^{(n)}_k) \, \delta_{nm}
   =  \sum_{k = 1}^4 \bar{u}^i_ku^j_k\delta_{nm} = \delta_{ij} \, \delta_{nm},
\end{aligned}
$$
and hence $\{\tilde{\chi}^j_n\}_{n\in\Nat^*}^{j\in\{1,\dots,4\}}$ 
is biorthonormal to $\{{\chi}^j_n\}_{n\in\Nat^*}^{j\in\{1,\dots,4\}}$,
where $\Nat^* := \Nat \setminus \{0\}$.
The construction of the biorthonormal pair for $n=0$ is similar.

\subsection{The asymptotic behaviour of the eigenprojectors}\label{sec:asymptotics}
In this subsection, we establish Theorem~\ref{Thm.proj}. 

For every $n \in \Nat$,
let 
$$
  \P^+_n := \sum_{i=1}^2\chi^j_n(\tilde{\chi}^j_n, \cdot)
  \qquad \mbox{and} \qquad
  \P^-_n := \sum_{i=3}^4\chi^j_n(\tilde{\chi}^j_n, \cdot)
$$ 
denote the spectral projectors onto the eigenspaces associated 
with $\sqrt{2n+m^2}$ and $-\sqrt{2n+m^2}$, respectively.
\begin{Proposition}
$\norm{\P^+_n} = \norm{\P^-_n}$.
\end{Proposition}
\begin{proof}
By virtue of the symmetry~\eqref{symmetry}, 
the action of the Dirac matrix~$\alpha_0$ on the eigenfunctions 
maps the eigenspace associated with the eigenvalue~$\lambda$ 
into the eigenspace associated with the eigenvalue~$-\lambda$.
In fact, one can easily verify that 
$\alpha_0\chi^1_n = \chi^4_n$ and $\alpha_0\chi^2_n = \chi^3_n$.
Then, from the unitarity and hermiticity of $\alpha_0$, one has
$$ 
  \alpha_0\P^+_n\alpha_0 = \sum_{j=1}^2\alpha_0\chi^j_n(\tilde{\chi}^j_n, \alpha_0\cdot) = \sum_{j=3}^4\chi^j_n(\tilde{\chi}^j_n,\cdot) = \P_n^-
  \,,
$$
and therefore the desired claim.
\end{proof}

Because of this proposition, without loss of generality,
we focus on the eigenprojectors corresponding to non-negative eigenvalues
and drop the superindex ${+}$ for the rest of this section. 

For the non-relativistic harmonic operator~$S_\theta$,
it is easy to see that 
$\norm{P_n} = \norm{\varphi_n} \, \norm{\tilde{\varphi}_n}$, where $P_n$ are the projections associated to the non-relativistic rotated harmonic oscillator, cf.\ Eq.~\eqref{projectors}.
Moreover, $\norm{\varphi_n}=\norm{\tilde{\varphi}_n}$.
Then~\eqref{projectors} follows by the asymoptotics of Hermite functions.

We split the proof of Theorem~\ref{Thm.proj} in three {lemmata}.
\begin{Lemma}\label{lem:orthogonalsubspace}
{For every $n\in\Nat^*$, $(\chi^1_n,\chi^2_n) = 0$
and $(\tilde{\chi}^1_n,\tilde{\chi}^2_n) = 0$.}
\end{Lemma}
\begin{proof}
{From definitions~\eqref{eq:coefficients_+} and~\eqref{eq:coefficients_-},} 
it follows that $\bar{u}^1_1u^2_1 + \bar{u}^1_3u^2_3 =\bar{u}^2_1u^2_2 + \bar{u}^1_4u^2_4 = 0$ and therefore
$$
  {({\chi}^1_n, {\chi}^2_n)} 
= |\bar{u}^1_1u^2_1 + \bar{u}^1_3u^2_3| \, \norm{\varphi_n}^2 
+ |\bar{u}^2_1u^2_2 + \bar{u}^1_4u^2_4| \, \norm{\varphi_{n-1}}^2=0.
$$
The proof for the pair $(\tilde{\chi}^1_n, \tilde{\chi}^2_n)$ is identical.
\end{proof}

\begin{Lemma}\label{Lem1}
{$
\displaystyle
  \limsup_{n\to\infty}\frac{\log{\norm{\P_n}}}{n}
  \leq \log \sqrt{\frac{1+|\sin\theta|}{1-|\sin\theta|}}
$.}
\end{Lemma}
\begin{proof}
Using the orthogonality proved in the previous lemma 
and the Schwarz inequality{,} 
it follows that
$$
  \|\P_n \psi\|^2 = \|\chi^1_n\|^2 \, |(\tilde{\chi}^1_n,\psi)|^2
  + \|\chi^2_n\|^2 \, |(\tilde{\chi}^2_n,\psi)|^2
  \leq \|\chi^1_n\|^2 \, \|\tilde{\chi}^1_n\|^2 \, \|\psi\|^2
  + \|\chi^2_n\|^2 \, \|\tilde{\chi}^2_n\|^2 \, \|\psi\|^2
  \,
$$
and therefore $\norm{\P_n}\leq\sqrt{ \|\chi^1_n\|^2 \|\tilde{\chi}^1_n\|^2 
  + \|\chi^2_n\|^2 \|\tilde{\chi}^2_n\|^2 }$.
Now {it follows from \eqref{explicit1}--\eqref{explicit2} 
and the normalisation $|u^j|=1$ in~$\Com^4$ that,
for each $j \in \{1,2\}$ and $n \in \Nat^*$,
$$
  \norm{\chi^j_n}^2 \leq
  2 \, (\norm{\varphi_n}^2 + \norm{\varphi_{n-1}}^2)
  \qquad \mbox{and} \qquad
  \norm{\tilde\chi^j_n}^2 \leq
  2 \, (\norm{\varphi_n}^2 + \norm{\varphi_{n-1}}^2)
  \,.
$$
}
Combining the previous inequalities we get that
$$
  \norm{\P_n}\leq \sqrt{8} \, (\norm{\varphi_n}^2 + \norm{\varphi_{n-1}}^2).
$$
Diving by $n$ at both sides, {recalling}~\eqref{projectors} 
and taking the limit {$n\to\infty$} proves the result.
\end{proof}

\begin{Lemma}\label{Lem2}
{$\displaystyle
  \liminf_{n\to\infty}\frac{\log{\norm{\P_n}}}{n}
  \geq \log \sqrt{\frac{1+|\sin\theta|}{1-|\sin\theta|}}
$.}
\end{Lemma}
\begin{proof}
{Given any $n \in \Nat^*$, one has 
$$
  \P_n \tilde{\chi}^1_n 
  = {\chi}^1_n ({\tilde\chi}^1_n, \tilde{\chi}^1_n)
  + {\chi}^2_n ({\tilde\chi}^2_n, \tilde{\chi}^1_n)
  = {\chi}^1_n \| \tilde\chi^1_n \|^2   
  \,,
$$
where the second equality follows by Lemma~\ref{lem:orthogonalsubspace}.
Consequently, recalling~\eqref{explicit1} and~\eqref{explicit2}, 
$$
  \|\P_n\| \geq \frac{\| \P_n \tilde{\chi}^1_n \|}{\| \tilde\chi^1_n \|}
  = \|{\chi}^1_n\| \,  \|\tilde\chi^1_n \|   
  = (|u^1_1|^2 + |u^1_3|^2 ) \, \norm{\varphi_n}^2 
  + (|u^1_2|^2 + |u^1_4|^2 ) \, \norm{\varphi_{n-1}}^2
  \,,
$$
where the last equality employs that 
$\norm{\varphi_n} = \norm{\tilde{\varphi}_n}$ for every $n \in \Nat$.
Using the concavity of the logarithm, we get}
\begin{align*}
  \frac{\log{\norm{\P_n}}}{n} 
  & \geq \frac{ \log
  \big[
  (|u^1_1|^2 + |u^1_3|^2) \, \norm{\varphi_n}^2 
  + (|u^1_2|^2 + |u^1_4|^2) \, \norm{\varphi_{n-1}}^2
  \big]
  }{n} 
  \\
    &\geq (|u^1_1|^2 + |u^1_3|^2 ) \,
    \frac{\log{\norm{\varphi_n}^2}}{n} 
    +  (|u^1_2|^2 + |u^1_4|^2 ) \,
    \frac{ \log \norm{\varphi_{n-1}}^2 }{n}
    \,.
\end{align*}
{Taking the limit $n\to\infty$,
recalling~\eqref{projectors} 
and the normalisation $|u^1|=1$ in~$\Com^4$, the result follows.}
\end{proof}

{Theorem~\ref{Thm.proj} follows as a consequence 
of Lemmata~\ref{Lem1} and~\ref{Lem2}.}

\subsection{The non-relativistic limit}
{The final part of this section is devoted 
to the analysis of the non-relativistic limit.} 
For this purpose it is needed to include the dimensionfull constants in the expressions of the differential operators. 
The relativistic rotated {harmonic oscillator} becomes 
\begin{equation}\label{eq:dimensionrelativistic}
  {\hat{H}}_\theta(c) := - ic \alpha_1 e^{-i\theta/2} \partial_x
  - cm\omega\alpha_2 e^{i\theta/2} x
  + mc^2 \alpha_3
  \,{,} \qquad  
  \Dom({\hat{H}}_\theta(c)) 
  {:= \Dom(H_\theta) = } \mathcal{D}^4   \,,
\end{equation} 
where $c$ is the speed of light, 
$m>0$ is the mass and $\omega$ is the frequency of the oscillator.
The non-relativistic rotated {harmonic oscillator reads}
$$
  {\hat{S}}_\theta 
  = \frac{1}{2m}\left(-e^{-i\theta}\partial^2_x 
  + m^2\omega^2 e^{i\theta}x^2 \right)
  \,{,} \qquad  
  \Dom({\hat{S}}_\theta) 
  {:= \Dom(S_\theta) = }   
  H^2(\Real) \cap \sii(\Real,x^4 \, \der x) \,.
$$

{Note that the mass is supposed to be positive in this subsection,
for massless particles are not covered by the non-relativistic quantum mechanics.
Mathematically, it is justified by the fact that the non-relativistic limit
corresponds to considering the limit $c \to \infty$
and no additive renormalisation is possible if $m=0$.}

{The following result shows that 
(suitably renormalised) $\hat{H}_\theta(c)$ 
converges as $c \to \infty$ 
to (a constant shift of)~$\hat{S}_\theta$ 
in a norm-resolvent sense.}

\begin{Theorem}\label{thm:nrlimit}
{For every $z\in\C\backslash\R$,
\begin{equation}\label{limit}
  \lim_{c\to\infty} 
  \left(
  \hat{H}_\theta(c)-mc^2 I_{\Com^4}  - z I_{\Com^4} 
  \right)^{-1} 
  = \left( \hat{S}_\theta \, I_{\Com^4} 
  + \frac{ \omega}{2} \, i\alpha_1\alpha_2 
  -z I_{\Com^4}
  \right)^{-1}
\end{equation}
}%
in the uniform operator topology.
\end{Theorem}
\begin{proof}
The proof is a direct application of 
{\cite[Thm.~6.1 \& Corol.~6.2]{Thaller} 
to} the operator ${\hat{H}}_\theta(c) = cQ -Mc^2\tau$, 
where 
$
  Q := - i\alpha_1 e^{-i\theta/2} \partial_x
  - m\omega\alpha_2 e^{i\theta/2} x
$
{with $\Dom(Q) := \mathcal{D}^4$}
is a closed operator with compact resolvent and real spectrum, 
$\tau {:=} \alpha_3$ and  $M {:=} m {I}_{\C^4}$. 
{The self-adjointness in the rather algebraic proof of~\cite{Thaller} 
does not play any fundamental role. 
Indeed, the argument just requires $z\in \rho(Q)\cap\rho(Q^2/2m)$,
which is guaranteed by the fact that the spectra of~$Q$ 
and~$Q^2$ are real.}
\end{proof}

{
\begin{Remark}
Since the Dirac operator~$H_\theta$ is not written
in the standard representation (in the sense that the mass
term is multiplied by~$\alpha_3$ instead of~$\alpha_4$),
the limiting operator in~\eqref{limit} has a less transparent structure
than in~\cite[Eq.~(6.18)]{Thaller}.
Nonetheless,
this representation is unitarily equivalent to the standard one 
and the limiting operator is still a diagonal operator commuting with
$P_+ := \frac{1}{2}(I_{\Com^4}+\alpha_3)$,
the orthogonal projection onto the positive subspace of~$\alpha_3$.
\end{Remark}
}

%--------------------------%
\section{The pseudospectra}\label{Sec.pseudo}
%--------------------------%
%
This section is devoted to the proof of Theorem~\ref{Thm.pseudo}.
 
Writing
\begin{equation*}%\label{upper}
\begin{aligned}
  \frac{1}{\eps^2} <
  \|(H_\theta^2-z^2)^{-1}\| 
  &= \|(H_\theta-z)^{-1} (H_\theta+z)^{-1}\| 
  \\
  &\leq \|(H_\theta-z)^{-1} \| \, \| (H_\theta+z)^{-1}\| 
  \\
  &= \| (H_\theta \pm z)^{-1}\|^2 
  \,,
\end{aligned}
\end{equation*}
where the last equality follows by~\eqref{pseudosymmetry},
we deduce that
\begin{equation*}
  z^2 \in \sigma_{\eps^2}(H_\theta^2)
  \quad \Longrightarrow \quad
  \pm z \in \sigma_{\eps}(H_\theta)
  \,.
\end{equation*}
That is,
\begin{equation*}%\label{inclusion.easy}
\left.
\begin{aligned}
  &z^2-m^2 - 1 \in \sigma_{\eps^2}(S_\theta)
  \\
  & \qquad\qquad \mbox{or}
  \\
  &z^2-m^2 + 1 \in \sigma_{\eps^2}(S_\theta)
\end{aligned}
\ \right\}
  \quad \Longrightarrow \quad
  \pm z \in \sigma_{\eps}(H_\theta)
  \,.
\end{equation*}
Combining this result with~\eqref{pseudo},
we arrive at the first inclusion of Theorem~\ref{Thm.pseudo}.
 
To establish the second inclusion of Theorem~\ref{Thm.pseudo},
we use the following theorem. 
  \begin{Theorem}[{\cite[Theorem~IV.1.16]{Kato}}] \label{thm:kato_resolvent_bound}
    Let $X, Y$ be two Banach spaces and let $T$ and $B$ be operators from $X$ to $Y$.
    Let $T^{-1}$ exist and be bounded.
    Assume that, for every $\phi \in X$, it holds
    \begin{equation*}
      \|B\phi\| \leq a \|\phi\| + b \|T\phi\|,
      \quad \text{with} \quad a \|T^{-1}\| + b < 1.
    \end{equation*}
    Then, $S := T + B$ is closed and invertible, with
    \begin{equation}\label{K.bound}
      \|S^{-1}\| \leq \frac{\|T^{-1}\|}{1 - a \|T^{-1}\| - b}.
    \end{equation}
  \end{Theorem}
We already used this theorem when defining~$H_\theta$ in Section~\ref{Sec.def}.
Now we employ the bound~\eqref{K.bound}.

  \begin{Proposition} \label{prop:resolvent_bound}
    Let $|\theta| < \frac{\pi}{2}$ and let $z$ be a complex number such that
    \begin{equation*}
      |\tan\mbox{$\frac{\theta}{2}$}| \, (m + |\Re z|) 
      < \left(1 - |\tan\mbox{$\frac{\theta}{2}$}| \right) |\Im z|.
    \end{equation*}
    Then 
    \begin{equation*}
      \|(H_\theta - z)^{-1}\| 
      \leq \frac{1}{(1 - |\tan \frac{\theta}{2}|) \, |\Im z| 
      - |\tan \frac{\theta}{2}| \, (m + |\Re z|)}.
    \end{equation*}
  \end{Proposition}
  \begin{proof}
    Let $A$, $B$ be the symmetric and anti-symmetric parts of $H_\theta$ 
    with $m=0$ as defined in~\eqref{eq:H_theta_AB}.
    Since $H_\theta - z = A + B + m\alpha_3 - z$, 
    the idea is to obtain the result by applying Theorem~\ref{thm:kato_resolvent_bound} with 
    $S := H_\theta$ and
    $T := A + m\alpha_3 - z$.
    The term $m\alpha_3 - \Re z$ is bounded and $A$ is self-adjoint, 
    and therefore $A + m\alpha_3 - \Re z$ is self-adjoint. 
    This implies that~$T$ is invertible with bounded inverse satisfying 
    $\|T^{-1}\| \leq 1/|\Im z|$.
    Moreover, $\alpha_3$ is unitary so, 
    for any $\psi \in \Dom(A) = {\mathcal{D}^4}$, it holds
    \begin{equation*}
      \|(A - i\Im z)\psi\| 
      \leq \|(A + m\alpha_3 - z)\psi\| + (m + |\Re z|)\|\psi\|.
    \end{equation*}
    Finally, by~\eqref{ineq.A}, \eqref{ineq.B} 
    and the conditions on~$\theta$ and $\Im z$, 
    it follows
    \begin{equation*}
      \|B\psi\| 
      \leq |\tan \mbox{$\frac{\theta}{2}$}| \, \|(A - i\Im z) \psi\|
      \leq |\tan \mbox{$\frac{\theta}{2}$}| \,
      \big(\|T\psi\| + (m + |\mu|) \, \|\psi\|\big).
    \end{equation*}
The announced bound then follows by~\eqref{K.bound}.
  \end{proof}

{To discuss the optimality of Theorem~\ref{Thm.pseudo}, 
we restrict to $m=0$.
Then the second inclusion of Theorem~\ref{Thm.pseudo}
(or Proposition~\ref{prop:resolvent_bound}) implies that 
\begin{equation} 
  \limsup_{r \to \pm\infty}
  \big\|\big(H_\theta -r e^{i\vartheta}\big)^{-1}\big\|  
  < \infty
  \qquad \mbox{whenever} \qquad
  |\vartheta| >  
  \arctan\left(
  \frac{|\tan\mbox{$\frac{\theta}{2}$}|}
  {1 - |\tan\mbox{$\frac{\theta}{2}$}|}
  \right)
  =: f(\theta)
  \,.
\end{equation}
From the second inclusion of Theorem~\ref{Thm.pseudo}
(see also~\eqref{wild}), we necessarily have $f(\theta) \geq |\frac{\theta}{2}|$.
This inequality becomes sharp in the limit $\theta \to 0$,
for $f(\theta) = |\frac{\theta}{2}| + O(\theta^2)$.
Therefore, Theorem~\ref{Thm.pseudo} becomes optimal for small values of~$\theta$.
On the other hand, $f(\theta) \to \frac{\pi}{2}$ as $\theta \to \frac{\pi}{2}$,
which shows that $f(\theta) \sim |\theta|$ in this limit.
Consequently, Theorem~\ref{Thm.pseudo} is not optimal 
for large values of~$|\theta|$.
The dependence of $f(\theta)$ on~$\theta$ 
is depicted in Figure~\ref{Fig.bounds}.
We leave as an open problem whether $|\vartheta| = |\frac{\theta}{2}|$
is the transition angle between the nature of pseudospectral properties
of~$H_\theta$ (in the same way as $|\vartheta| = |\theta|$ is 
in the non-relativistic setting of~$S_\theta$,
see~\eqref{Davies.wild} versus~\eqref{Davies.nowild}).}

  \begin{figure}[h!]
    \begin{center}
      \includegraphics[width=0.7\textwidth]{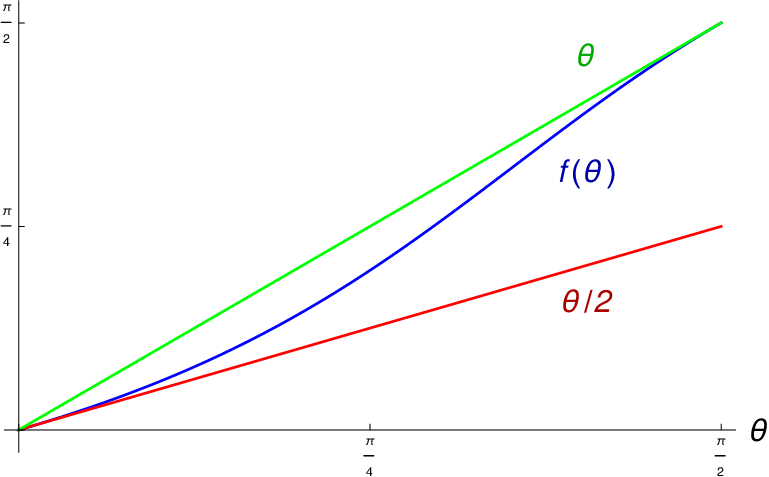}
    \end{center}
    \caption{Visualisation of the (lack of) optimality of Theorem~\ref{Thm.pseudo}
    for $m=0$.}\label{Fig.bounds}
  \end{figure}

\subsection*{Keywords}
pseudospectrum, resolvent estimate, relativistic harmonic oscilator,
non-self-adjoint Dirac operator, rotated harmonic oscillator, Davies operator, wild basis properties

\subsection*{MCS 2020}
34L40, 34L15, 47A10, 81Q12 

\subsection*{Acknowledgment}
D.K.\ was supported
by the EXPRO grant No.~20-17749X
of the Czech Science Foundation.
  A.B. and J.M.P.P. acknowledge support provided by the ``Agencia Estatal de Investigación (AEI)'' Research Project PID2020-117477GB-I00, by the QUITEMAD Project TEC-2024/COM-84-QUITEMAD-CM funded by the Madrid Government (Comunidad de Madrid-Spain) and by the Madrid Government (Comunidad de Madrid-Spain) under the Multiannual Agreement with UC3M in the line of ``Research Funds for Beatriz Galindo Fellowships'' (C\&QIG-BG-CM-UC3M), and in the context of the V PRICIT (Regional Programme of Research and Technological Innovation).
  A.B.\ acknowledges financial support from the Spanish Ministry of Universities through the UC3M Margarita Salas 2021-2023 program (``Convocatoria de la Universidad Carlos III de Madrid de Ayudas para la recualificación del sistema universitario español para 2021-2023'').
%

%\newpage
%\vfill  
%--------------%
% BIBLIOGRAPHY %
%--------------%
%
%\addcontentsline{toc}{section}{References}
\bibliography{bib}
\bibliographystyle{amsplain}

\end{document}